\documentclass{strippedconm-p-l}

\usepackage{amssymb}

\newtheorem{theorem}{Theorem}[section]
\newtheorem{lemma}[theorem]{Lemma}

\theoremstyle{definition}
\newtheorem{definition}[theorem]{Definition}
\newtheorem{example}[theorem]{Example}

\theoremstyle{remark}

\numberwithin{equation}{section}

\begin{document}

\title[Analogs of Uniform Algebras]{A Survey of Non-Complex Analogs of Uniform Algebras}


%    author one information
\author{J. W. D. Mason}
\address{School of Mathematical Sciences, University of Nottingham, Nottingham, NG7 2RD, UK}
\email{pmxjwdm@nottingham.ac.uk}
\thanks{The author was supported by a PhD grant from the EPSRC (UK)}

%    The 2010 edition of the Mathematics Subject Classification is
%    now available.  If you are citing a classification from the
%    new scheme, use the following input coding instead.
%\subjclass[2010]{Primary }

\date{September 01, 2010}

\begin{abstract}
We survey commutative and non-commutative analogs of uniform algebras in the Archimedean settings and also offer some non-Archimedean examples. Constraints on the development of non-complex uniform algebras are also discussed.
\end{abstract}

\maketitle
\begin{center}
Accepted to appear in the Proceedings of the Sixth Conference on Function Spaces, AMS Contemporary Mathematics book series.
\end{center}
\section*{Introduction}
\label{sec:Intro}
Uniform algebras have been extensively investigated because of their importance in the theory of uniform approximation and as examples of complex Banach algebras.  As enquiry broadens one may ask whether analogous algebras exist when a complete valued field other than the complex numbers is used as the underlying field of scalars over which the algebra is a vector space.  In many situations the analogous algebras obtained are without qualifying subalgebras. For example if $F$ is a complete valued field other than the complex numbers then the Stone-Weierstrass theorem, or its non-Archimedean generalisation by Kaplansky, shows that the Banach $F$-algebra $C_{F}(X)$ of all continuous $F$-valued functions on a compact Hausdorff space $X$, which we also require to be totally disconnected in the non-Archimedean setting, is without a proper subalgebra that satisfies the conditions of the theorem, see section \ref{sec:Swei}. However, in the commutative, Archimedean setting, Kulkarni and Limaye in a paper from 1981, \cite{Kulkarni-Limaye1981}, introduced the now familiar theory of real function algebras which provides examples with proper qualifying subalgebras. One important departure in the definition of real function algebras from that of complex uniform algebras is that they are real Banach algebras of continuous complex-valued functions. Other appropriately analogous algebras of continuous functions that take values in some complete valued field or division ring extending the field of scalars over which the algebra is a vector space also exist. One benefit of establishing such analogs of uniform algebras is that they increase the variety of complete normed algebras that we know can be represented by algebras of continuous functions. 

In this paper we survey commutative and non-commutative results in the Archimedean settings and also offer some non-Archimedean examples. The results on generalising uniform algebras over all complete valued fields, as the single theory presented at the 6th Conference on Function Spaces at SIU-Edwardsville, will be presented in the author's thesis and later publications.
\section{Constraints on the development of non-complex uniform algebras}
\label{sec:Cons}
\subsection{Complete valued fields}
\label{sec:Comp}
For non-complex uniform algebras we need to consider which fields have a complete valuation i.e. a complete multiplicative norm. The trivial valuation taking values in $\{0,1\}$ can be applied to any field $F$ and induces the trivial topology. It is the most basic example of a complete non-Archimedean valuation which are those complete valuations that satisfy the strong triangle inequality, 
\begin{equation*}
|a-b|\leq\mbox{max}(|a|,|b|)\mbox{ for all }a,b\in F.
\label{equ:Stin}
\end{equation*} 
If the valuation on a valued field is non-Archimedean then we call the valued field a non-Archimedean field, else we call the valued field Archimedean. Every complete Archimedean field is path-connected whereas every complete non-Archimedean field is totally disconnected. It turns out that almost all complete valued fields are non-Archimedean with $\mathbb{R}$ and $\mathbb{C}$ being the only two Archimedean exceptions up to isomorphism as topological fields. This in part follows from the Gel'fand-Mazur Theorem, see \cite{Schikhof} for details. 

There are examples of complete non-Archimedean fields of non-zero characteristic with non-trivial valuation. For each there is a prime $p$ such that the field is a transcendental extension of the finite field $\mathbb{F}_{p}$ of $p$ elements. One example of this sort is the valued field of formal Laurent series $\mathbb{F}_{p}\{\{T\}\}$ in one variable over $\mathbb{F}_{p}$ with termwise addition, multiplication in the form of the Cauchy product and valuation given at zero by $|0|_{T}:=0$ and on the units $\mathbb{F}_{p}\{\{T\}\}^{*}$ by,
\begin{equation*}
\mbox{$|\sum_{n\in\mathbb{Z}}a_{n}T^{n}|_{T}$}:=r^{-\mbox{min}\{n:a_{n}\not= 0\}}\mbox{ for any fixed }r>1.
\label{equ:Tval}
\end{equation*}
We note here that the only valuation on a finite field is the trivial valuation.

On the other hand complete valued fields of characteristic zero necessarily contain one of the completions of the rational numbers $\mathbb{Q}$. The Levi-Civita field, see \cite{Shamseddine}, is such a valued field. A total order can be put on the Levi-Civita field such that the order topology agrees with the topology induced by the field's valuation which is non-trivial. This might be useful to those interested in generalising the theory of C*-algebras to new fields where there is a need to define positive elements. The completion of $\mathbb{Q}$ that the Levi-Civita field contains is in fact $\mathbb{Q}$ itself since the valuation when restricted to $\mathbb{Q}$ is trivial. 

The non-trivial completions of the rational numbers are given by Ostrowski's Theorem, see \cite{Schikhof}.
\begin{theorem}
Each non-trivial valuation on the field of rational numbers is equivalent either to the absolute value function or to one of the $p$-adic valuations for some prime $p$.
\label{thr:Ostr}
\end{theorem}
Concerning $p$-adic theory, for a given prime $p$, each element of $\mathbb{Q}$ can be uniquely expanded in the form of a formal Laurent series evaluated at $p$ with integer coefficients belonging to $\{0,1, \cdots, p-1\}$. For example for $p=5$ we have,
\begin{equation*}
\frac{1}{2}=3\cdot5^{0}+2\cdot5+2\cdot5^{2}+2\cdot5^{3}+2\cdot5^{4}+\cdots.
\label{equ:Padc}
\end{equation*}
Similar to the case of the decimal expansion, a $p$-adic expansion of an element of $\mathbb{Q}$ will have a term such that the coefficients of all later terms, to its right, form a repeating sequence. For a given prime $p$ the $p$-adic valuation on $\mathbb{Q}$ is given at zero by $|0|_{p}:=0$ and at any unit $a=\sum_{n\in\mathbb{Z}}a_{n}p^{n}\in\mathbb{Q}^{*}$ by,
\begin{equation*}
|a|_{p}:=p^{-\nu_{p}(a)}\mbox{ where }\nu_{p}(a):=\mbox{min}\{n:a_{n}\not= 0\}.
\label{equ:Pval}
\end{equation*}
The function $\nu_{p}$ is often extended to zero by $\nu_{p}(0):=+\infty$ and may be referred to as the valuation logarithm. Denoting the trivial valuation on $\mathbb{Q}$ by $|\cdot|_{\infty}$ and the absolute value function by $|\cdot|_{0}$, the different valuations on $\mathbb{Q}$, when restricted to the units $\mathbb{Q}^{*}$, are related by the equation,
\begin{equation*}
|\cdot|_{0}=\frac{1}{|\cdot|_{2}|\cdot|_{3}|\cdot|_{5}\cdots|\cdot|_{\infty}}.
\label{equ:Allv}
\end{equation*}
For a given prime $p$ the valued field $\mathbb{Q}_{p}$ of $p$-adic numbers is the completion of $\mathbb{Q}$ with respect to the valuation $|\cdot|_{p}$ and it is the valued field of all formal Laurent series evaluated at $p$ with integer coefficients belonging to $\{0,1, \cdots, p-1\}$. For each prime $p$, $\mathbb{Q}_{p}$ is locally compact. Moreover as a field, rather than as a valued field, $\mathbb{Q}_{p}$ has an embedding into $\mathbb{C}$. The $p$-adic valuation on $\mathbb{Q}_{p}$ can then be prolongated to a complete valuation on the complex numbers which in this case as a valued field we denote as $\mathbb{C}_{p}$. Unlike $\mathbb{C}$ however, $\mathbb{C}_{p}$ is not locally compact. 

More generally, and unlike in the Archimedean setting, we have the following theorem. 
\begin{theorem}
Every complete non-Archimedean field $F$ has a non-trivial extension $L$ for which the complete valuation on $F$ can be prolongated to a complete valuation on $L$. 
\label{thr:Krul}
\end{theorem}
Theorem \ref{thr:Krul} follows from Krull's existence theorem and the fact that every valued field has a completion, see \cite{McCarthy} and \cite{Schikhof} for details. With the above details on complete valued fields at hand we now progress towards a survey of non-complex uniform algebras.
\subsection{The Stone-Weierstrass theorem and its generalisations}
\label{sec:Swei}
Also see Section \ref{sec:Ncya}. We recall the following basic definition and note that introductions to uniform algebras can be found in \cite{Browder}, \cite{Gamelin} and \cite{Stout}.
\begin{definition}
Let $C_{\mathbb{C}}(X)$ be the algebra of all continuous complex valued functions defined on a compact, Hausdorff space $X$. A {\em{uniform algebra}}, $A$, is a subalgebra of $C_{\mathbb{C}}(X)$ that is complete with respect to the sup norm, contains the constant functions and separates the points of $X$ in the sense that for all $x_{1}, x_{2}\in X$ with $x_{1}\not= x_{2}$ there is $f\in A$ satisfying $f(x_{1})\not= f(x_{2})$.
\label{def:Ualg}
\end{definition}
Some authors take Definition \ref{def:Ualg} to be a representation of uniform algebras and take a uniform algebra $A$ to be a complex unital Banach algebra with square preserving norm, $\|a^{2}\|=\|a\|^{2}$ for all $a\in A$. This is quite legitimate since the Gelfand transform shows us that every such algebra is isometrically isomorphic to an algebra conforming to Definition \ref{def:Ualg}. In this paper we mainly consider non-complex analogs of Definition \ref{def:Ualg} but acknowledge the importance of representation results, many of which can be found in the references.

The most obvious non-complex analog of Definition \ref{def:Ualg} is obtained by simply replacing the complex numbers in the definition by some other complete valued field $F$. In this case, whilst $C_{F}(X)$ will be complete and contain the constants, we need to take care concerning the topology on $X$ when $F$ is non-Archimedean. 
\begin{theorem}
Let $F$ be a complete, non-Archimedean, valued field and let $C_{F}(X)$ be the algebra of all continuous $F$-valued functions defined on a compact, Hausdorff space $X$. Then $C_{F}(X)$ separates the points of $X$ if and only if $X$ is totally disconnected.
\label{thr:Xtop}
\end{theorem}
For clarification of terminology from general topology used in this section see \cite{Willard}. Before giving a proof of Theorem \ref{thr:Xtop} we have the following version of Urysohn's lemma.
\begin{lemma}
Let $X$ be a totally disconnected, compact, Hausdorff space with $\{x,y_{1},y_{2},y_{3}, \cdots, y_{n}\}\subseteq X$, $x\not= y_{i}$ for all $i\in\{1, \cdots, n\}$ where $n\in\mathbb{N}$. Let $L$ be any non-empty topological space and $a,b\in L$. Then there exists a continuous map $h:X\longrightarrow L$ such that $h(x)=a$ and $h(y_{1})=h(y_{2})=h(y_{3})= \cdots =h(y_{n})=b$.
\label{lem:Urys}
\end{lemma}
\begin{proof}
Since $X$ is a Hausdorff space, for each $i\in\{1, \cdots, n\}$ there are disjoint open subsets $U_{i}$ and $V_{i}$ of $X$ with $x\in U_{i}$ and $y_{i}\in V_{i}$. Hence $U:=\bigcap_{i\in\{1, \cdots, n\}}U_{i}$ is an open subset of $X$ with $x\in U$ and $U\cap V_{i}=\emptyset$ for all $i\in\{1, \cdots, n\}$. Now since $X$ is a totally disconnected, compact, Hausdorff space, $x$ has a neighborhood base of clopen sets, see Theorem 29.7 of \cite{Willard}. Hence there is a clopen subset $W$ of $X$ with $x\in W\subseteq U$. The function $h:X\longrightarrow L$ given by $h(W):=\{a\}$ and $h(X\backslash W):=\{b\}$ is continuous as required.
\end{proof}
We now give the proof of Theorem \ref{thr:Xtop}.
\begin{proof}
With reference to Lemma \ref{lem:Urys} it remains to show that $C_{F}(X)$ separates the points of $X$ only if $X$ is totally disconnected. Let $X$ be a compact, Hausdorff space such that $C_{F}(X)$ separates the points of $X$. Let $U$ be a non-empty connected subset of $X$ and let $f\in C_{F}(X)$. We note that $f(U)$ is a connected subset of $F$ since $f$ is continuous. Now, since $F$ is non-Archimedean it is totally disconnected i.e. its connected subsets are singletons. Hence $f(U)$ is a singleton and so $f$ is constant on $U$. Therefore, since $C_{F}(X)$ separates the points of $X$, $U$ is a singleton and $X$ is totally disconnected.
\end{proof}
We next consider the constraints on $C_{F}(X)$ revealed by the Stone-Weierstrass theorem and its generalisations. Recall that in the complex case, for suitable $X$, there exist uniform algebras that are proper subalgebras of $C_{\mathbb{C}}(X)$. However if $A$ is such a uniform algebra then $A$ is not self-adjoint, that is there is $f\in A$ with $\bar{f}\notin A$. 
\begin{example}
A standard example is the {\em{disc algebra}} $A(\bar{\Delta})\subseteq C_{\mathbb{C}}(\bar{\Delta})$, of functions analytic on $\Delta:=\{z\in\mathbb{C}:|z|<1\}$, which is as far from being self-adjoint as possible since if both $f$ and $\bar{f}$ are in $A(\bar{\Delta})$ then $f$ is constant, see \cite{Browder}, \cite{Kulkarni-Limaye1992} or \cite{Stout}. Furthermore by Mergelyan's theorem, see \cite{Gamelin}, or otherwise we have $P(\bar{\Delta})=A(\bar{\Delta})$ where $P(\bar{\Delta})$ is the uniform algebra of all functions on $\bar{\Delta}$ that can be uniformly approximated by polynomials restricted to $\bar{\Delta}$ with complex coefficients.
\label{exa:Dalg}
\end{example}
There are many interesting examples involving Swiss cheese sets, see \cite{Feinstein-Heath}, \cite{Gamelin}, \cite{Mason} and \cite{Roth}.

In the real case the Stone-Weierstrass theorem for $C_{\mathbb{R}}(X)$ says that for every compact Hausdorff space $X$, $C_{\mathbb{R}}(X)$ is without a proper subalgebra that is uniformly closed, contains the real numbers and separates the points of $X$. A proof can be found in \cite{Kulkarni-Limaye1992}.

We close this section by considering the non-Archimedean case which is given by a theorem of Kaplansky, see \cite{Berkovich} or \cite{Kaplansky}.
\begin{theorem}
Let $F$ be a complete non-Archimedean valued field, let $X$ be a totally disconnected compact Hausdorff space, and let $A$ be a $F$-subalgebra of $C_{F}(X)$ which satisfies the following conditions:
\begin{enumerate}
\item[\textup{(a)}]
the elements of $A$ separate the points of $X$;
\item[\textup{(b)}]
for each $x\in X$ there exists $f\in A$ with $f(x)\not=0$.
\end{enumerate}
Then $A$ is everywhere dense in $C_{F}(X)$.
\label{thr:Kapl}
\end{theorem}
Note that, in Theorem \ref{thr:Kapl}, $A$ being a $F$-subalgebra of $C_{F}(X)$ means that $A$ is a subalgebra of $C_{F}(X)$ and a vector space over $F$. If we take $A$ to be unital then condition (b) in Theorem \ref{thr:Kapl} is automatically satisfied and the theorem is analogous to the real version of the Stone-Weierstrass theorem. 

In section \ref{sec:Ncom} we will see that real function algebras are a useful example when considering non-complex analogs of uniform algebras with qualifying subalgebras.
\section{Non-complex analogs of uniform algebras}
\label{sec:Ncom}
\subsection{The commutative, Archimedean case}
\label{sec:Ycya}
Real function algebras were introduced by Kulkarni and Limaye in a paper from 1981, see \cite{Kulkarni-Limaye1981}. For our purposes we recall their definition and consider an example. For a thorough text on the theory see \cite{Kulkarni-Limaye1992}. 
\begin{definition}
Let $X$ be a compact Hausdorff space and $\tau$ a topological involution on $X$, i.e. a homeomorphism with $\tau\circ\tau =\mbox{id}$ on $X$. A {\em{real function algebras on}} $(X,\tau)$ is a real subalgebra $A$ of
\begin{equation*}
C(X,\tau):=\{f\in C_{\mathbb{C}}(X):f(\tau(x))=\bar{f}(x)\mbox{ for all }x\in X\}
\end{equation*}
that is complete with respect to the sup norm, contains the real numbers and separates the points of $X$.
\label{def:Refa}
\end{definition}
Note that $C(X,\tau)$ in Definition \ref{def:Refa} is itself a real function algebra on $(X,\tau)$ and in some sense it is to real function algebras as $C_{\mathbb{C}}(X)$ is to uniform algebras. The following example is standard, see \cite{Kulkarni-Limaye1992}.
\begin{example}
Recall from Example \ref{exa:Dalg} the disc algebra $A(\bar{\Delta})$ on the closed unit disc and let $\tau:\bar{\Delta}\longrightarrow\bar{\Delta}$ be the map $\tau(z):=\bar{z}$ given by complex conjugation, which we note is a Galois automorphism on $\mathbb{C}$. Now let 
\begin{equation*}
B(\bar{\Delta}):=A(\bar{\Delta})\cap C(\bar{\Delta},\tau). 
\end{equation*}
We see that $B(\bar{\Delta})$ is complete since both $A(\bar{\Delta})$ and $C(\bar{\Delta},\tau)$ are, and similarly $B(\bar{\Delta})$ contains the real numbers. Further by the definition of $C(\bar{\Delta},\tau)$ and the fact that $A(\bar{\Delta})=P(\bar{\Delta})$ we have that $B(\bar{\Delta})$ is the $\mathbb{R}$-algebra of all uniform limits of polynomials on $\bar{\Delta}$ with real coefficients. Hence $B(\bar{\Delta})$ separates the points of $\bar{\Delta}$ since it contains the function $f(z):=z$. However whilst $\tau$ is in $C(\bar{\Delta},\tau)$ it is not an element of $A(\bar{\Delta})$. Therefore $B(\bar{\Delta})$ is a real function algebra on $(\bar{\Delta},\tau)$ and a proper subalgebra of $C(\bar{\Delta},\tau)$. It is referred to as the {\em{real disc algebra}}. 
\label{exa:Rdal}
\end{example}
\subsection{The commutative, non-Archimedean case}
\label{sec:Ycna}
If $J$ is a maximal ideal of a commutative unital complex Banach algebra $A$ then $J$ has codimension one since $A/J$ with the quotient norm is isometrically isomorphic to the complex numbers by, in part, the Gelfand-Mazur theorem. In contrast, for a complete non-Archimedean field $F$, if $I$ is a maximal ideal of a commutative unital Banach $F$-algebra then $I$ may have large finite or infinite codimension, note Theorem \ref{thr:Krul}. Hence, with Gelfand transform theory in mind, it makes sense to consider non-Archimedean analogs of uniform algebras in the form suggested by real function algebras where the functions take values in a complete extension of the underlying field of scalars. Moreover when there is a lattice of intermediate fields then these fields provide a way for a lattice of extensions of the algebra to occur. See \cite{Berkovich} and \cite{Escassut} for one form of the Gelfand transform in the non-Archimedean setting.

For the purpose of this survey we consider commutative, non-Archimedean examples involving an order two extension of the scalars. More generally we choose a homeomorphisom $\tau$ with order dividing that of the Galois automorphisom used.
\begin{example}
Let $F:=\mathbb{Q}_{5}$ and $L:=\mathbb{Q}_{5}(\sqrt{2})$. Suppose towards a contradiction that $\sqrt{2}$ is already an element of $\mathbb{Q}_{5}$. With reference to section \ref{sec:Comp}, we would have $1=|2|_{5}=|\sqrt{2}^{2}|_{5}=|\sqrt{2}|_{5}^{2}$ giving $|\sqrt{2}|_{5}=1$. But then $\sqrt{2}$ would have a 5-adic expansion of the form $\sum_{n\in\mathbb{N}_{0}}a_{n}5^{n}$ with $a_{0}\not=0$ where $\mathbb{N}_{0}:=\mathbb{N}\cup\{0\}$. Hence
\begin{equation}
a_{0}^{2} +2a_{0}\sum_{n\in\mathbb{N}}a_{n}5^{n}+\left(\sum_{n\in\mathbb{N}}a_{n}5^{n}\right)^{2}
\label{equ:sqrt}
\end{equation}
should be equal to 2. In particular $a_{0}^{2}$ should have the form $2+b$ where $b$ is a natural number, with a factor of 5, that cancels with the remaining terms of (\ref{equ:sqrt}). But since $a_{0}\in\{1,2,3,4\}$ we have $a_{0}^{2}\in\{1,4,4+5,1+3\cdot5\}$, a contradiction. Now the Galois group of $^{L}/_{F}$ is $\mbox{Gal}(^{L}/_{F})=\{\mbox{id},g\}$ where $g$ sends $\sqrt{2}$ to $-\sqrt{2}$. The complete valuation on $F$ has a unique prolongation to a complete valuation on $L$, see \cite{Schikhof} and Theorem \ref{thr:Krul}. Hence $g$ is an isometry on $L$ and so explicitly we have, for all $a\in L$,
\begin{equation*}
|a|_{L}=\sqrt{|a|_{L}|g(a)|_{L}}=\sqrt{|ag(a)|_{L}}=\sqrt{|ag(a)|_{5}},
\end{equation*}
noting that $ag(a)\in F$. In terms of the valuation logarithm from Section \ref{sec:Comp} we have $\sqrt{|ag(a)|_{5}}=5^{-\frac{1}{2}\nu_{5}(ag(a))}$ and so we define the valuation logarithm for $L$ as $\omega(a):=\frac{1}{2}\nu_{5}(ag(a))$ for all $a\in L$. One can check using the 5-adic expansion that for any element $a+\sqrt{2}b\in L^{*}$, for $a,b\in F$, we have $\omega(a+\sqrt{2}b)=\frac{1}{2}\nu_{5}(a^{2}-2b^{2})\in\mathbb{Z}$. In fact $^{L}/_{F}$ is classified as an unramified extension meaning that $\omega(L^{*})=\omega(F^{*})$. Now $L$ is locally compact and so $\bar{\Delta}_{L}:=\{x\in L:|x|_{L}\leq1\}$ is compact. Further if we take $\tau_{1}$ to be the restriction of $g$ to $\bar{\Delta}_{L}$ then, since $g$ is an isometry on $L$, $\tau_{1}$ is a topological involution on $\bar{\Delta}_{L}$ and
\begin{equation*}
C(\bar{\Delta}_{L},\tau_{1},g):=\{f\in C_{L}(\bar{\Delta}_{L}):f(\tau_{1}(x))=g(f(x))\mbox{ for all }x\in\bar{\Delta}_{L}\} 
\end{equation*}
is a non-Archimedean analog of the real disc algebra. In fact $C(\bar{\Delta}_{L},\tau_{1},g)$ is an $F$-algebra such that every polynomial and power series in $C(\bar{\Delta}_{L},\tau_{1},g)$ has $F$-valued coefficients noting that the polynomials on $\bar{\Delta}_{L}$ with $L$-valued coefficients are uniformly dense in $C_{L}(\bar{\Delta}_{L})$ by Theorem \ref{thr:Kapl}. The proof that $C(\bar{\Delta}_{L},\tau_{1},g)$ is uniformly closed and separates the points of $\bar{\Delta}_{L}$ is the same as that for real function algebras, see \cite{Kulkarni-Limaye1992}, only we use Lemma \ref{lem:Urys} in place of Urysohn's lemma. 
\label{exa:Fdal}
\end{example}
\begin{example}
Let $F$, $L$, $\bar{\Delta}_{L}$, $\omega$ and $g$ be as in Example \ref{exa:Fdal}. Define $\tau_{2}(0):=0$ and for $x\in \bar{\Delta}_{L}\backslash\{0\}$, 
\begin{equation*}
\tau_{2}(x):= \left\{ \begin{array} {r@{\quad\mbox{if}\quad}l}
5x & 2\mid\omega(x) \\
5^{-1}x & 2\nmid\omega(x).
\end{array} \right. 
\end{equation*}
One checks that $\tau_{2}$ is a topological involution on $\bar{\Delta}_{L}$. In this case the only power series in $C(\bar{\Delta}_{L},\tau_{2},g)$ are constants belonging to $F$. However there are elements of $C(\bar{\Delta}_{L},\tau_{2},g)$ that can be expressed as power series on each of the circles $\mathcal{C}(n):=\{x\in\bar{\Delta}_{L}:\omega(x)=n\}$, for $n\in\omega(\bar{\Delta}_{L})$, but for such an element we can not use the same power series on all of these circles.
\label{exa:Ycna}
\end{example}
\subsection{The non-commutative, Archimedean case}
\label{sec:Ncya}
In recent years a theory of non-commutative real function algebras has been developed by Jarosz and Abel, see \cite{Abel-Jarosz} and \cite{Jarosz}. The continuous functions involved take values in Hamilton's real quaternions, $\mathbb{H}$, which are an example of a non-commutative complete Archimedean division ring and $\mathbb{R}$-algebra. Viewing $\mathbb{H}$ as a real vector space, the valuation on $\mathbb{H}$ is the Euclidean norm which is complete, Archimedean and indeed a valuation since being multiplicative on $\mathbb{H}$. To put $\mathbb{H}$ into context, as in the case of complete Archimedean fields, there are very few unital division algebras over the reals with the Euclidean norm as a valuation. Up to isomorphism they are $\mathbb{R}$, $\mathbb{C}$, $\mathbb{H}$ and the octonions $\mathbb{O}$. We note that the octonions are non-associative. The proof that there are no other unital division algebras over the reals with the Euclidean norm as a valuation is given by Hurwitz's 1, 2, 4, 8 Theorem, see \cite{Shapiro} and \cite{Lewis}. In particular for such an algebra $\mathbb{A}$ the square of the Euclidean norm is a regular quadratic form on $\mathbb{A}$ and since for $\mathbb{A}$ the Euclidean norm is a valuation it is multiplicative. This shows that $\mathbb{A}$ is a real composition algebra to which Hurwitz's 1, 2, 4, 8 Theorem can be applied.

Here we only briefly consider non-commutative real function algebras and hence the reader is also referred to \cite{Jarosz}. Note the author of this paper is unaware of any such developments involving the octonions. Here is Jarosz's analog of Definition \ref{def:Refa}.
\begin{definition}
Let $\mbox{Gal}(^{\mathbb{H}}/_{\mathbb{R}})$ be the group of all automorphisms on $\mathbb{H}$ that are the identity on $\mathbb{R}$. Let $X$ be a compact space and $\mbox{Hom}(X)$ be a group of homeomorphisms on $X$. For a group homomorphism $\Phi:\mbox{Gal}(^{\mathbb{H}}/_{\mathbb{R}})\longrightarrow\mbox{Hom}(X)$, $\Phi(T)=\Phi_{T}$, we define
\begin{equation*}
C_{\mathbb{H}}(X,\Phi):=\{f\in C_{\mathbb{H}}(X):f(\Phi_{T}(x))=T(f(x))\mbox{ for all }x\in X\mbox{ and } T\in\mbox{Gal}(^{\mathbb{H}}/_{\mathbb{R}})\}.
\end{equation*}
\label{def:Ncrf}
\end{definition}
\vspace{-3mm}
The groups $\mbox{Gal}(^{\mathbb{H}}/_{\mathbb{R}})$ and $\mbox{Hom}(X)$ in Definition \ref{def:Ncrf} have composition as their group operation. As a conjecture the author suggests that Definition \ref{def:Ncrf} may also be useful if $\mbox{Gal}(^{\mathbb{H}}/_{\mathbb{R}})$ is replaced by a subgroup, particularly when considering extensions of the algebra. However Definition \ref{def:Ncrf} has already been used successfully in the representation of real unital Banach algebras with square preserving norm. Jarosz also gives the following Stone-Weierstrass theorem type result. 
\begin{definition}
A real algebra $A$ is {\em{fully non-commutative}} if every nonzero multiplicative, linear functional $F:A\longrightarrow\mathbb{H}$ is surjective.
\label{def:Fnon}
\end{definition}
\begin{theorem}
Let $A$ be a fully non-commutative closed subalgebra of $C_{\mathbb{H}}(X)$. Then $A=C_{\mathbb{H}}(X)$ if and only if $A$ strongly separates the points of $X$ in the sense that for all $x_{1}, x_{2}\in X$ with $x_{1}\not= x_{2}$ there is $f\in A$ satisfying $0\not= f(x_{1})\not= f(x_{2})$.
\label{thr:Ncsw}
\end{theorem}
\subsection{The non-commutative, non-Archimedean case}
\label{sec:Ncna}
Non-commutative, non-Archimedean analogs of uniform algebras have yet to be seen. Hence in this section we give an example of a non-commutative extension of a complete non-Archimedean field which would be appropriate when considering such analogs of uniform algebras. We first have the following definition from the general theory of quaternion algebras. The main reference for this section is \cite{Lam} but \cite{Lewis} is also useful.
\begin{definition}
Let $F$ be a field, with characteristic not equal to 2, and $s,t\in F^{*}$ where $s=t$ is allowed. We define the {\em{quaternion $F$-algebra}} $(\frac{s,t}{F})$ as follows. As a 4-dimensional vector space over $F$ we define 
\begin{equation*}
\left(\frac{s,t}{F}\right):=\{a+bi+cj+dk:a,b,c,d\in F\}
\end{equation*}
with $\{1,i,j,k\}$ as a natural basis giving the standard coordinate-wise addition and scalar multiplication. As an $F$-algebra, multiplication in $(\frac{s,t}{F})$ is given by 
\begin{equation*}
i^{2}=s,\mbox{\hspace{3mm}}j^{2}=t,\mbox{\hspace{3mm}}k^{2}=ij=-ji
\end{equation*}
together with the usual distributive law and multiplication in $F$. 
\label{def:Gtqa}
\end{definition}
Hamilton's real quaternions, $\mathbb{H}:=(\frac{-1,-1}{\mathbb{R}})$ with the Euclidean norm, is an example of a non-commutative, complete valued, Archimedean, division algebra over $\mathbb{R}$. It is not the case that every quaternion algebra $(\frac{s,t}{F})$ will be a division algebra, although there are many examples that are. For our purposes we have the following example.
\begin{example}
Using $\mathbb{Q}_{5}$, the complete non-Archimedean field of 5-adic numbers, define
\begin{equation*}
\mathbb{H}_{5}:=\left(\frac{5,2}{\mathbb{Q}_{5}}\right). 
\end{equation*}
Then for $q,r\in\mathbb{H}_{5}$, $q=a+bi+cj+dk$, the conjugation on $\mathbb{H}_{5}$ given by
\begin{equation*}
\bar{q}:=a-bi-cj-dk 
\end{equation*}
is such that $\overline{q+r}=\bar{q}+\bar{r}$, $\overline{qr}=\bar{r}\bar{q}$, $\bar{q}q=q\bar{q}=a^{2}-5b^{2}-2c^{2}+10d^{2}$ with $\bar{q}q\in\mathbb{Q}_{5}$ and
\begin{equation*}
|q|_{\mathbb{H}_{5}}:=\sqrt{|\bar{q}q|_{5}} 
\end{equation*}
is a complete non-Archimedean valuation on $\mathbb{H}_{5}$, where $|\cdot|_{5}$ is the 5-adic valuation on $\mathbb{Q}_{5}$. In particular $\mathbb{H}_{5}$, together with $|\cdot|_{\mathbb{H}_{5}}$, is an example of a non-commutative, complete valued, non-Archimedean, division algebra over $\mathbb{Q}_{5}$. When showing this directly it is useful to know that for $a,b,c,d\in\mathbb{Q}_{5}$ we have
\begin{equation*}
\nu_{5}(a^{2}-5b^{2}-2c^{2}+10d^{2})=\mbox{min}\{\nu_{5}(a^{2}), \nu_{5}(5b^{2}), \nu_{5}(2c^{2}), \nu_{5}(10d^{2})\} 
\end{equation*}
where $\nu_{5}$ is the 5-adic valuation logarithm as defined in Section \ref{sec:Comp}.
\label{exa:Ncna}
\end{example}
More generally for the $p$-adic field $\mathbb{Q}_{p}$ the quaternion algebra $(\frac{p,u}{\mathbb{Q}_{p}})$ will be a division algebra as long as $u$ is a unit of $\{a\in\mathbb{Q}_{p}:|a|_{p}\leq 1\}$, i.e. $|u|_{p}=1$, and $\mathbb{Q}_{p}(\sqrt{u})$ is a quadratic extension of $\mathbb{Q}_{p}$.
%\nocite{*}
%\bibliographystyle{amsplain}
%\bibliography{masonAMSbibdata}
\providecommand{\bysame}{\leavevmode\hbox to3em{\hrulefill}\thinspace}
\providecommand{\MR}{\relax\ifhmode\unskip\space\fi MR }
\providecommand{\MRhref}[2]{%
  \href{http://www.ams.org/mathscinet-getitem?mr=#1}{#2}
}
\providecommand{\href}[2]{#2}

\section*{Acknowledgments}
The author would like to thank Dr. K. Jarosz, as the principal organiser of the 6th Conference on Function Spaces (2010) at Southern Illinois University Edwardsville, for arranging funds in support of the author's participation at the conference.

Similarly the author is grateful for the conference funding provided by the School of Mathematical Sciences, University of Nottingham, UK.

Special thanks to Dr. J. F. Feinstein for his valuable advice and enthusiasm as the author's supervisor.
\end{document}